\documentclass[11pt]{amsart}
\usepackage{xcolor,mathtools}

\include{package}
\definecolor{grn}{rgb}{0,0.6,0}
\definecolor{mrn}{rgb}{0.3,0,0}
\definecolor{blue}{rgb}{0,0,0.7}
\definecolor{Mygray}{rgb}{0.75,0.75,0.75}
\definecolor{auburn}{rgb}{0.43, 0.21, 0.1}
\definecolor{britishracinggreen}{rgb}{0.0, 0.26, 0.15}
\definecolor{taupe}{rgb}{0.28, 0.24, 0.2}
\usepackage{amsmath,amssymb}
\newtheorem{theorem}{Theorem}[section]

\newtheorem{proposition}{Proposition}[section]

\newtheorem{lemma}{Lemma}[section]
\newtheorem{defn}{Definition}[section]

\newtheorem{remark}{Remark}[section]

\newtheorem*{ack}{Acknowledgements}

\usepackage{latexsym,amssymb}
\usepackage{amssymb}
\usepackage{graphicx}
\usepackage{ textcomp }
\usepackage{tikz}
\usetikzlibrary{shapes.geometric, arrows}
\tikzstyle{startstop} = [rectangle, rounded corners, minimum width=3cm, minimum height=1cm,text centered, draw=black, fill=white!30]
\tikzstyle{io} = [trapezium, trapezium left angle=70, trapezium right angle=110, minimum width=3cm, minimum height=1cm, text centered, draw=black, fill=white!30]
\tikzstyle{process} = [rectangle, minimum width=2cm, minimum height=1cm, text centered, draw=black, fill=white!30]
\tikzstyle{decision} = [rectangle, minimum width=1cm, minimum height=1cm, text centered, draw=black, fill=white!30]
\tikzstyle{arrow} = [thick,->,>=stealth]
\usepackage[left=2cm,right=2cm,top=2cm,bottom=2cm]{geometry}

\allowdisplaybreaks

\begin{document}
\baselineskip=14.5pt
\title[Bi-quadratic P\'{o}lya fields]{Bi-quadratic P\'{o}lya fields with five distinct ramified primes}

\author{Md. Imdadul Islam, Jaitra Chattopadhyay and Debopam Chakraborty}
\address[Md. Imdadul Islam and Debopam Chakraborty]{Department of Mathematics, BITS-Pilani, Hyderabad campus, Hyderabad, INDIA}
\address[Jaitra Chattopadhyay]{Department of Mathematics, Siksha Bhavana, Visva-Bharati, Santiniketan - 731235, West Bengal, India}

\email[Md. Imdadul Islam]{p20200059@hyderabad.bits-pilani.ac.in}

\email[Jaitra Chattopadhyay]{jaitra.chattopadhyay@visva-bharati.ac.in; chat.jaitra@gmail.com} 

\email[Debopam Chakraborty]{debopam@hyderabad.bits-pilani.ac.in}

\begin{abstract}
For an algebraic number field $K$, the P\'{o}lya group of $K$, denoted by $Po(K),$ is the subgroup of the ideal class group $Cl_{K}$ generated by the ideal classes of the products of prime ideals of same norm. The number field $K$ is said to be P\'{o}lya if $Po(K)$ is trivial. Motivated by several recent studies on the group $Po(K)$ when $K$ is a totally real bi-quadratic field, we investigate the same with five distinct odd primes ramifying in $K/\mathbb{Q}$. This extends the previous results on this problem, where the number of distinct ramified primes was at most four.  
\end{abstract}

\maketitle

\section{introduction}

Let $K$ be an algebraic number field with ring of integers $\mathcal{O}_{K}$. Let $Cl_{K}$ be the ideal class group of $K$ with class number $h_{K}$. For every rational prime $p$ and integer $f \geq 1$, The notion of \textit{Ostrowski ideals} $\Pi_{p^{f}}(K)$ of $K$ was defined in \cite{ost}, in the following way.
\begin{equation*}
\displaystyle\Pi_{p^{f}}(K) = \displaystyle\prod_{\substack {\mathfrak{p} \in {\rm{Max}}(\mathcal{O}_{K})\\ N_{K/\mathbb{Q}} (\mathfrak{m}) = p^{f}}} \mathfrak{m}. 
\end{equation*}
$K$ is called a \textit{P\'{o}lya field} (cf. \cite{pol}, \cite{zan}) if all the Ostrowski ideals $\Pi_{p}^{f}(K)$ for arbitrary prime powers are principal. To measure the extent of failure of the Ostrowski ideals from being principal, the notion of {\it P\'{o}lya group} is defined as follows. 
\begin{defn} ((cf. \cite{cah1}, \textsection II.4))\label{def0}
  The P\'{o}lya group of a number field $K$ is the subgroup $Po(K)$ of the ideal class group $Cl(K)$ generated by the ideal classes of all the Ostrowski ideals $\Pi_{p}^{f}(K)$.  
\end{defn}
    Works of Brumer and Rosen \cite{bru} showed that for a finite Galois extension $K/\mathbb{Q}$, the Ostrowski ideals freely generate the ambiguous ideals. Therefore, in this case, $Po(K)$ can also be regarded as a subgroup of $Cl(K)$ generated by the classes of ambiguous ideals.  An alternative definition of P\'{o}lya field arises from the behavior of \textit{integer-valued polynomials} over $\mathcal{O}_{K}$ in the following way. Let $$\text{Int}(\mathcal{O}_{K}) = \{f(X) \in K[X] \text{ } : f(\mathcal{O}_{K}) \subseteq \mathcal{O}_{K}\}$$ be the ring of integer-valued polynomials on $K$. Due to the works of Cahen et. al. (cf. \cite{cah1}, \cite{cah2}), it is known that $\text{Int}(\mathcal{O}_{K})$ is a free $\mathcal{O}_{K}$-module. A \textit{regular basis} of $\text{Int}(\mathcal{O}_{K})$ is an $\mathcal{O}_{K}$-basis $\{f_{i}\}_{i \geq 0}$ such that degree$(f_{i}) = i$ for all $i \geq 0$. 

\begin{defn} \cite[Definition II.4.1]{cah1}
     $K$ is said to be a P\'{o}lya field if the $\mathcal{O}_{K}$-module $\text{Int}(\mathcal{O}_{K})$ has a regular basis. 
\end{defn}

It is evident from Definition \ref{def0} that ramified primes play a significant role in determining the P\'{o}lya group $Po(K)$ of a number field $K$. For quadratic fields, the following result of Hilbert is one of the earliest results related to P\'{o}lya groups. The result was later generalized in \cite[Proposition 3.1]{zan} by Zantema for arbitrary finite Galois extensions of $\mathbb{Q}$ . 

\begin{proposition}\cite[Theorem 106]{hil}
    Let $K$ be a quadratic field and denote the number of ramified primes in $K/ \mathbb{Q}$ by $s_{K}$. If $K$ is real and the fundamental unit of $K$ is of norm $1$, then $|Po(K)| = 2^{s_{K}-2}$. Otherwise we have $|Po(K)| = 2^{s_{K}-1}$.
\end{proposition}

    Recently, works of Chabert (cf. \cite{cha}) and Maarefparvar et. al. (cf. \cite{mar3}) generalized the above result independently, introducing concepts like \textit{relative P\'{o}lya group} along the way. While the \textit{class number one problem} is still unresolved for real quadratic fields, there are many results regarding trivial P\'{o}lya group for number fields of small degrees. Quadratic P\'{o}lya fields have been characterized through the work of Zantema in \cite{zan}. Later, in \cite{ler}, Leriche did a similar characterization for cyclic cubic and cyclic quartic P\'{o}lya fields. In that same work, Leriche also classified bi-quadratic fields that are compositum of two quadratic P\'{o}lya fields \cite[Theorem 5.4]{ler}. 
    
\smallskip

The size of $Po(K)$ for a totally real bi-quadratic field has been studied extensively in recent years. In \cite{hei1}, Heideryan and Rajaei discussed bi-quadratic P\'{o}lya fields with exactly one quadratic P\'{o}lya subfield. In \cite{hei2}, the same authors constructed non-P\'{o}lya bi-quadratic fields having low ramification. P\'{o}lya groups of small degree number fields were also discussed in \cite{mar1}, \cite{mar2} and \cite{tou}. The second author of this article and Saikia constructed totally real bi-quadratic fields with large P\'{o}lya groups (cf. \cite{cha1}) and determined $Po(K)$ for some families of bi-quadratic fields $K$ with a non-principal Euclidean ideal class (cf. \cite{cha2}). Borrowing the notion of {\it consecutive number fields}, originally proposed by Iizuka in \cite{iiz} for imaginary quadratic fields, the authors of this article have recently constructed simplest cubic and consecutive bi-quadratic fields with large P\'{o}lya groups in \cite{isl}. 

\smallskip

It is well-known that if the number of ramified primes in a totally real bi-quadratic field $K$ is more than five, then $Po(K)$ is non-trivial. Thus, the classification of totally real bi-quadratic field P\'{o}lya fields only deals with at most five ramifications. Many results are available when the number is less than five, but not many with exactly five ramifications. The following result due to Maarefparvar and Rajaei is concerned with bi-quadratic fields with precisely four distinct ramified primes.
    \begin{proposition}\cite[Theorem 4.1]{mar3}
        Let $K = \mathbb{Q}(\sqrt{pq}, \sqrt{qrs})$, where $p,q,r$ and $s$ are prime numbers with $p \equiv 3 \pmod 8$, $q \equiv s \equiv 5 \pmod 8$ and $r \equiv 7 \pmod 8$. Then $Po(K) \simeq (\mathbb{Z}/ 2 \mathbb{Z})^{t}$ for $t \in \{0,1,2\}$. Furthermore, depending on various conditions on the Legendre symbols among these primes, each case can occur. 
    \end{proposition}
In this paper, we deal with totally real bi-quadratic fields with precisely five distinct odd ramified primes. We provide a complete classification for such P\'{o}lya bi-quadratic fields where the discriminant of one quadratic subfield comprises of three distinct odd primes. More precisely, we prove the following theorem.
\begin{theorem}\label{mainthm}
Let $p_{1} \equiv p_{2} \equiv q_{1} \equiv q_{2} \equiv 3 \pmod 4$, $p_{3} \equiv 1 \pmod 4$ are distinct prime numbers such that $q_{1}q_{2} \equiv p_{1}p_{2}p_{3} \equiv 1 \pmod 8$ and let $K = \mathbb{Q}(\sqrt{q_{1}q_{2}}, \sqrt{p_{1}p_{2}p_{3}})$. Then $K$ is a P\'{o}lya field in each of the following cases.
\begin{enumerate}
\item $\big(\frac{p_{1}}{p_{2}}\big) = 1$, $\big(\frac{p_{1}}{p_{3}}\big) = 1$, $\big(\frac{p_{2}}{p_{3}}\big) = -1$, $\big(\frac{p_{3}}{q_{1}}\big) = \big(\frac{p_{3}}{q_{2}}\big) = \big(\frac{p_{1}}{q_{1}}\big) = \big(\frac{p_{1}}{q_{2}}\big) = -1$ and $\big(\frac{q_{1}q_{2}}{p_{2}}\big) = -1$. \\
\item $\big(\frac{p_{1}}{p_{2}}\big) = 1$, $\big(\frac{p_{1}}{p_{3}}\big) = -1$, $\big(\frac{p_{2}}{p_{3}}\big) = 1$, $\big(\frac{p_{3}}{q_{1}}\big) = \big(\frac{p_{3}}{q_{2}}\big) = - \big(\frac{p_{1}}{q_{1}}\big) = - \big(\frac{p_{1}}{q_{2}}\big) = -1,$ and $\big(\frac{q_{1}q_{2}}{p_{2}}\big) = -1$. \\
\item $\big(\frac{p_{1}}{p_{3}}\big) = -1$, $\big(\frac{p_{2}}{p_{3}}\big) = -1$, $\big(\frac{p_{3}}{q_{1}}\big) = - \big(\frac{p_{3}}{q_{2}}\big) = 1$ and $\big(\frac{p_{1}p_{2}}{q_{1}}\big) = \big(\frac{p_{1}p_{2}}{q_{2}}\big) = - \big(\frac{q_{1}q_{2}}{p_{1}}\big) = -1$. \\
\item $\big(\frac{p_{1}}{p_{2}}\big) = -1$, $\big(\frac{p_{1}}{p_{3}}\big) = 1$, $\big(\frac{p_{2}}{p_{3}}\big) = -1$, $\big(\frac{p_{3}}{q_{1}}\big) = \big(\frac{p_{3}}{q_{2}}\big) = -1$ and $\big(\frac{p_{1}p_{2}}{q_{1}}\big) = \big(\frac{p_{1}p_{2}}{q_{2}}\big) = - \big(\frac{q_{1}q_{2}}{p_{1}}\big) = 1$. \\
\item $\big(\frac{p_{1}}{p_{2}}\big) = -1$, $\big(\frac{p_{1}}{p_{3}}\big) = -1$, $\big(\frac{p_{2}}{p_{3}}\big) = 1$,$\big(\frac{p_{3}}{q_{1}}\big) = \big(\frac{p_{3}}{q_{2}}\big) = \big(\frac{p_{1}}{q_{1}}\big) = \big(\frac{p_{1}}{q_{2}}\big) = -1$ and $\big(\frac{q_{1}q_{2}}{p_{2}}\big) = - 1$.
\end{enumerate}
Moreover, if $ \big(\frac{p_{1}}{p_{3}}\big) = \big(\frac{p_{2}}{p_{3}}\big) = 1$, then either $K$ is a P\'{o}lya field or $Po(K) \simeq \mathbb{Z}/ 2 \mathbb{Z}$.  
\end{theorem}

\begin{remark}\label{rem0}
First, we note the well-known fact that if a real bi-quadratic field $K$ has more than or equal to six distinct rational primes that ramify in $K/ \mathbb{Q}$, then $Po(K)$ is always non-trivial. Here, five distinct odd primes $p_{1}, p_{2}, p_{3}, q_{1}$ and $q_{2}$ already ramify in $K/ \mathbb{Q}$. Hence, for $K/ \mathbb{Q}$ to be a P\'{o}lya field, the rational prime $2$ should not ramify in that extension. This is the reason behind focusing on solely the case $p_{1}p_{2}p_{3} \equiv q_{1}q_{2} \equiv 1 \pmod 4$.

\end{remark}
\begin{remark}
It is a simple application of Chinese remainder theorem to assert the existence of infinitely many primes satisfying the congruence and Legendre symbol conditions mentioned in the statement of Theorem \ref{mainthm}.
\end{remark}

\section{preliminaries}

For a finite Galois extension $K/\mathbb{Q}$, the Galois group $G \coloneqq {\rm{Gal}}(K/ \mathbb{Q})$ induces a $G$-module structure on the multiplicative group of units $\mathcal{O}_{K}^{*}$ of $\mathcal{O}_{K}$ via the action $(\sigma, \alpha) \mapsto \sigma(\alpha)$. In \cite{zan}, Zatema used this to study the relation between the ramified primes in $K/ \mathbb{Q}$ and the cohomology group $H^{1}(G, {\mathcal{O}_{K}^{*}})$, in the context of the size of the P\'{o}lya group $Po(K)$. We record his result as follows. 
\begin{proposition}\label{prop-zan}\cite[Page 163]{zan}
    Let $K$ be a finite Galois extension of $\mathbb{Q}$ with Galois group $G = {\rm{Gal}}(K/ \mathbb{Q})$. Let $p_{1}, p_{2}, ... , p_{t}$ be the ramified primes in $K/\mathbb{Q}$ with ramification index of $p_{i}$ being $e_{i}$ for $i \in \{1,2,...,t\}$. Then there is an exact sequence 
    \begin{equation}\label{exact-equn}
0 \to H^{1}(G,\mathcal{O}_{K}^{*}) \to \displaystyle\bigoplus_{i = 1}^{t}\mathbb{Z}/e_{i}\mathbb{Z} \to Po(K) \to 0.
\end{equation}
of abelian groups.
\end{proposition}
For a totally real bi-quadratic field $K$, the following two results help us understand the group $H^{1}(G, \mathcal{O}_{K}^{*})$ better. This is done by first identifying $H^{1}(G, \mathcal{O}_{K}^{*})$ with its $2$-torsion subgroup $H^{1}(G, \mathcal{O}_{K}^{*})[2]$ under certain constraints. The second result describes a convenient method of computing $H^{1}(G, \mathcal{O}_{K}^{*})[2]$.
\begin{lemma}\label{lem-zan1}\cite[Theorem 4]{set}
    Let $K_{1}, K_{2}$ and $K_{3}$ be the quadratic subfields of a totally real bi-quadratic field $K$. Then the index of $H^{1}(G, \mathcal{O}_{K}^{*})[2]$ in $H^{1}(G, \mathcal{O}_{K}^{*})$ is at most $2$ and the index equals $2$ if and only if the rational prime $2$ is totally ramified in $K/ \mathbb{Q}$ and there exists $\alpha_{i} \in K_{i}$ for each $i = 1,2,3$ with $N(\alpha_{1}) = N(\alpha_{2}) = N(\alpha_{3}) = \pm 2. $
\end{lemma}
For a non-zero integer $t$, we denote its canonical image in $\mathbb{Q}^{*}/(\mathbb{Q}^{*})^{2}$ by $[t]$ and for non-zero integers $t_{1},\ldots,t_{r}$, we denote the subgroup of $\mathbb{Q}^{*}/(\mathbb{Q}^{*})^{2}$ generated by the images of $t_{1},\ldots,t_{r}$ by $\langle [t_{1}],\ldots,[t_{r}] \rangle$.
\begin{lemma}\label{lem-zan2}\cite[Lemma 4.3]{zan}
Let $K_{1}, K_{2}$ and $K_{3}$ be the quadratic subfields of a totally real bi-quadratic field $K$. For $i=1,2,3$, let $\Delta_{i}$ be the square-free part of the discriminant of $K_{i}$ and let $u_{i} = x_{i}+y_{i}\sqrt{\Delta_{i}}$ be a fundamental unit of $\mathcal{O}_{K_{i}}$ for $i=1,2,3$, with $x_{i} > 0$. Then $H^{1}(G, \mathcal{O}_{K}^{*})[2]$ is isomorphic to the subgroup $\langle [\Delta_{1}], [\Delta_{2}], [\Delta_{3}], [a_{1}], [a_{2}], [a_{3}] \rangle$ of $\mathbb{Q}^{*}/ (\mathbb{Q}^{*})^{2}$ where for $i=1,2,3$, $a_{i} = N(u_{i}+1)$ if $N(u_{i}) = 1$, and $1$ otherwise.
\end{lemma}

\begin{remark}\label{rem1}
    From proposition \ref{prop-zan}, Lemma \ref{lem-zan1} and Lemma \ref{lem-zan2}, we have $Po(K) \cong \dfrac{\displaystyle\bigoplus_{i = 1}^{m}\mathbb{Z}/e_{i}\mathbb{Z}}{H^{1}(G,\mathcal{O}_{K}^{*})[2]}$ for a totally real bi-quadratic field $K$, where the rational prime $2$ is not totally ramified. Also, the presence of at least one prime that is congruent to $3$ modulo $4$ in the factorization of each $\Delta_{i}$ for $i=1,2,3$ ensures that $N(u_{i}) = 1$. We note that any other consideration will require a similar length of computations or less, which is evident from Lemma \ref{lem-zan2}.  
\end{remark}
For $K = \mathbb{Q}(\sqrt{q_{1}q_{2}}, \sqrt{p_{1}p_{2}p_{3}})$, it is evident that the numerator in the formula for $Po(K)$ is isomorphic to $(\mathbb{Z}/ 2\mathbb{Z})^{5}$. In the computation of $H^{1}(G,\mathcal{O}_{K}^{*})[2] = \langle{[\Delta_{1}], [\Delta_{2}], [\Delta_{3}], [a_{1}], [a_{2}], [a_{3}]\rangle}$, it is evident that $[\Delta_{1}] = [q_{1}q_{2}], [\Delta_{2}] = [p_{1}p_{2}p_{3}]$ and $[\Delta_{3}] = [q_{1}q_{2}p_{1}p_{2}p_{3}]$. In the following lemmas, we now focus on the computation of $[a_{1}]$ and $[a_{2}]$.
\begin{lemma}\label{lem0}
     Let $ K_{1} = \mathbb{Q}(\sqrt{q_{1} q_{2}})$ where $q_{1}$ and $q_{2}$ are prime numbers with $q_{1} \equiv q_{2} \equiv 3 \pmod 4$ such that $q_{1}q_{2} \equiv 1 \pmod 8$ and let $u_{1}$ denote a fundamental unit of $K_{1}$. Let $a_{1} = N(u_{1} + 1)$. Then $[a_{1}]$ is either $[q_{1}]$ or $[q_{2}]$.
\end{lemma}
\begin{proof}
We first note that if $u_{1} = \frac{x + y\sqrt{q_{1}q_{2}}}{2}$ for some odd integers $x$ and $y$, then $N(u_{1}) = \pm {1}$ implies that $x^{2} - y^{2}q_{1}q_{2} = \pm {4}$. Since $q_{1}q_{2} \equiv 1 \pmod {8}$, reading this equation modulo $8$ yields $\pm {4} \equiv x^{2} - y^{2}q_{1}q_{2} \equiv 0 \pmod {8}$, a contradiction. Thus we may assume that $u_{1} = x + y \sqrt{q_{1}q_{2}}$ for some $x,y \in \mathbb{Z}$.  Since $q_{1} \equiv q_{2} \equiv 3 \pmod 4$, we have $N(u_{1}) = 1$, which also implies that $N(u_{1}+1) = 2(x+1)$.

\smallskip

Moreover, $N(u_{1})= 1$ also implies that $x$ is odd and $y$ is even. Therefore, $\frac{x+1}{2} \cdot \frac{x-1}{2} = q_{1}q_{2}(\frac{y}{2})^{2}$ and since $\gcd(\frac{x+1}{2}, \frac{x-1}{2}) = 1$, we conclude that $\frac{x+1}{2} \in \{m^{2}, q_{1}m^{2}, q_{2}m^{2}, q_{1}q_{2}m^{2}\}$ where $m$ is an arbitrary factor of $\frac{y}{2}$. Denoting $n = \frac{y/2}{m}$, we now prove that $\frac{x+1}{2} \not\in \{m^{2}, q_{1}q_{2}m^{2}\}$.

\smallskip

If $\frac{x+1}{2} = m^{2}$, then $1 = \frac{x+1}{2} - \frac{x-1}{2}$ implies that $1 = m^{2} - q_{1}q_{2}n^{2}$. Therefore, $m+n\sqrt{q_{1}q_{2}} \in \mathcal{O}_{K_{1}}^{*}$, a contradiction to the fact that $u_{1}$ is a fundamental unit of $K_{1}$. Again, if $\frac{x+1}{2} = q_{1}q_{2}m^{2}$, then $1 = q_{1}q_{2}m^{2} - n^{2}$ implies that $\Big(\frac{-1}{q_{1}}\Big) = 1$, a contradiction as $q_{1} \equiv 3 \pmod 4$. Consequently, $\frac{x+1}{2} \in \{q_{1}m^{2}, q_{2}m^{2}\}$. The result now follows at once from the observation that $$[a_{1}] = [N(u_{1} + 1)] = [2(x + 1)] = \big[4 \cdot \frac{x + 1}{2}\big] = \big[\frac{x + 1}{2}\big] = [q_{1}] \mbox{ or } [q_{2}].$$
\end{proof}

\begin{lemma}\label{lem1}
Let $ K_{2} = \mathbb{Q}(\sqrt{p_{1} p_{2} p_{3}})$ where $p_{1}, p_{2}$ and $p_{3}$ are prime numbers with $p_{1} \equiv p_{2} \equiv 3 \pmod 4$, $p_{3} \equiv 1 \pmod 4$ and such that $p_{1}p_{2}p_{3} \equiv 1 \pmod 8$. Let $u_{2} = x_{2} + y_{2} \sqrt{p_{1}p_{2}p_{3}}$ be a fundamental unit of $K_{2}$ and let $a_{2} = N(u_{2} + 1)$. Then the square-free representative of $[a_{2}]$ is a proper, non-trivial divisor of $p_{1}p_{2}p_{3}$. More precisely, we have the following.
\begin{enumerate}
\item If $\big(\frac{p_{1}}{p_{2}}\big) = \big(\frac{p_{1}}{p_{3}}\big) = 1$, then $[a_{2}] = [p_{1}], [p_{3}]$ or $[p_{1}p_{3}]$. Moreover, if $\Big(\frac{p_{2}}{p_{3}}\Big) = -1$, then $[a_{2}] = [p_{1}]$. \\
\item If $\big(\frac{p_{1}}{p_{2}}\big) = \big(\frac{p_{1}p_{2}}{p_{3}}\big) = -1$, then $[a_{2}] = [p_{2}]$. \\
\item If $\big(\frac{p_{1}}{p_{3}}\big) = \big(\frac{p_{2}}{p_{3}}\big) = 1$, then $[a_{2}] = [p_{3}]$. \\
\item If $\big(\frac{p_{1}}{p_{3}}\big) = \big(\frac{p_{2}}{p_{3}}\big) =  -1$, then $[a_{2}] = [p_{1}p_{2}]$. \\
\item If $\big(\frac{p_{1}}{p_{2}}\big) = \big(\frac{p_{2}}{p_{3}}\big) = 1$, then $[a_{2}] = [p_{1}p_{3}]$ . \\
\item If $\big(\frac{p_{1}}{p_{2}}\big) = - \big(\frac{p_{1}}{p_{3}}\big) = -1$, then $[a_{2}] = [p_{2}], [p_{3}]$ or $[p_{2}p_{3}]$. Moreover, if $\Big(\frac{p_{2}}{p_{3}}\Big) = -1$, then $[a_{2}] = [p_{2}p_{3}]$.
\end{enumerate}
\end{lemma}
The proofs of different parts of Lemma \ref{lem1} are quite similar to each other. Therefore, for the sake of brevity, we furnish a detailed argument only for the first part. We first provide a pictorial description of the lemma to facilitate understanding of all the cases.

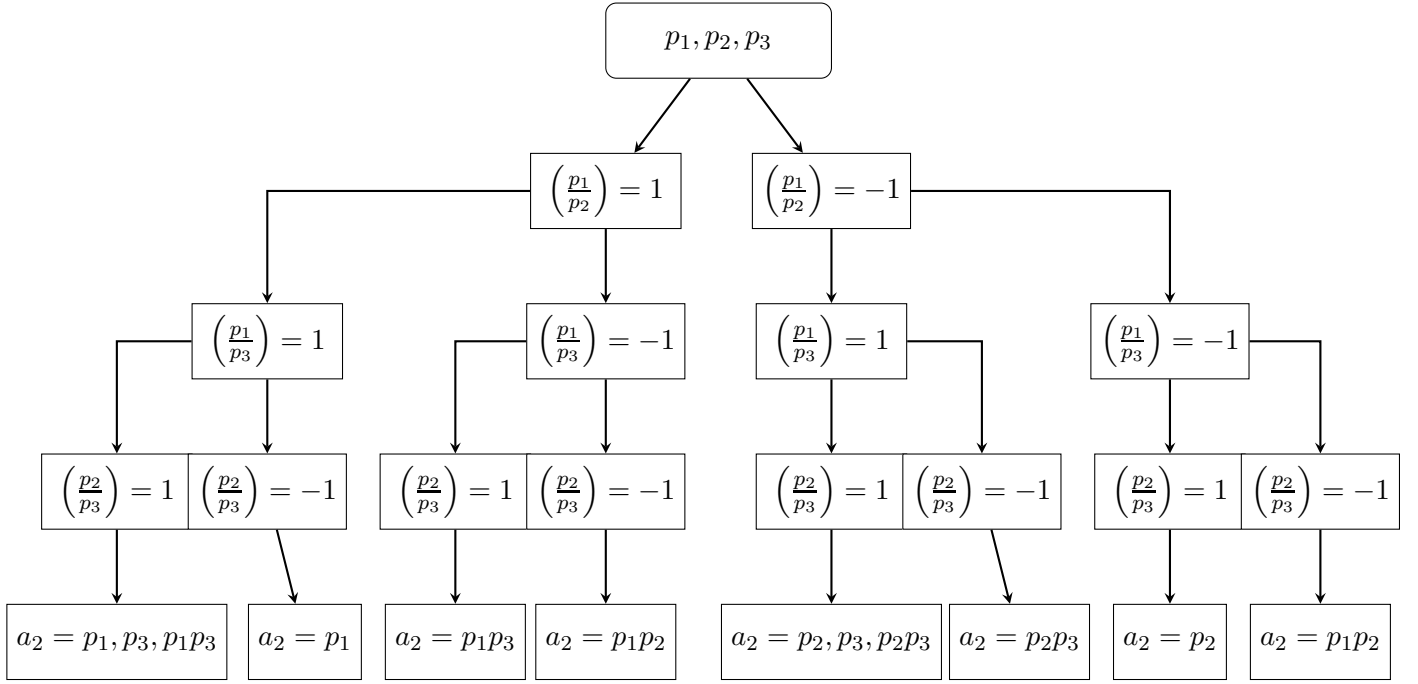
\begin{center}
\begin{tikzpicture}[node distance=2cm]
\node (start) [startstop] {$p_{1}, p_{2}, p_{3}$};
\node (pro2a) [process, below of=start, xshift=-1.5cm] {$\Big(\frac{p_{1}}{p_{2}}\Big) = 1$};
\node (pro2b) [process, below of=start, xshift=1.5cm] {$\Big(\frac{p_{1}}{p_{2}}\Big) = -1$};
\node (pro3a) [process, below of=pro2a, xshift=-4.5cm] {$\Big(\frac{p_{1}}{p_{3}}\Big) = 1$};
\node (pro3b) [process, below of=pro2a, xshift=0cm] {$\Big(\frac{p_{1}}{p_{3}}\Big) = -1$};
\node (pro3c) [process, below of=pro2b, xshift=0cm] {$\Big(\frac{p_{1}}{p_{3}}\Big) = 1$};
\node (pro3d) [process, below of=pro2b,xshift=4.5cm] {$\Big(\frac{p_{1}}{p_{3}}\Big) = -1$};
\node (pro4a) [process, below of=pro3a, xshift=-2cm] {$\Big(\frac{p_{2}}{p_{3}}\Big) = 1$};
\node (pro4b) [process, below of=pro3a, xshift=0cm] {$\Big(\frac{p_{2}}{p_{3}}\Big) = -1$};
\node (pro4c) [process, below of=pro3b, xshift=-2cm] {$\Big(\frac{p_{2}}{p_{3}}\Big) = 1$};
\node (pro4d) [process, below of=pro3b, xshift=0cm] {$\Big(\frac{p_{2}}{p_{3}}\Big) = -1$};
\node (pro4e) [process, below of=pro3c, xshift=-0cm] {$\Big(\frac{p_{2}}{p_{3}}\Big) = 1$};
\node (pro4f) [process, below of=pro3c, xshift=2cm] {$\Big(\frac{p_{2}}{p_{3}}\Big) = -1$};
\node (pro4g) [process, below of=pro3d, xshift=0cm] {$\Big(\frac{p_{2}}{p_{3}}\Big) = 1$};
\node (pro4h) [process, below of=pro3d, xshift=2cm] {$\Big(\frac{p_{2}}{p_{3}}\Big) = -1$};
\node (dec1) [decision, below of=pro4a, xshift=0cm] {$a_{2} = p_{1}, p_{3}, p_{1}p_{3}$};
\node (dec2) [decision, below of=pro4b, xshift=0.5cm] {$a_{2} = p_{1}$};
\node (dec3) [decision, below of=pro4c, xshift=0cm] {$a_{2} = p_{1}p_{3}$};
\node (dec4) [decision, below of=pro4d, xshift=0cm] {$a_{2} = p_{1}p_{2}$};
\node (dec5) [decision, below of=pro4e, xshift=0cm] {$a_{2} = p_{2}, p_{3}, p_{2}p_{3}$};
\node (dec6) [decision, below of=pro4f, xshift=0.5cm] {$a_{2} = p_{2}p_{3}$};
\node (dec7) [decision, below of=pro4g, xshift=0cm] {$a_{2} = p_{2}$};
\node (dec8) [decision, below of=pro4h, xshift=0cm] {$a_{2} = p_{1}p_{2}$};
\draw [arrow] (start) -- (pro2a);
\draw [arrow] (start) -- (pro2b);
\draw [arrow] (pro2a) -| (pro3a);
\draw [arrow] (pro2a) -- (pro3b);
\draw [arrow] (pro2b) -- (pro3c);
\draw [arrow] (pro2b) -| (pro3d);
\draw [arrow] (pro3a) -| (pro4a);
\draw [arrow] (pro3a) -- (pro4b);
\draw [arrow] (pro3b) -| (pro4c);
\draw [arrow] (pro3b) -- (pro4d);
\draw [arrow] (pro3c) -- (pro4e);
\draw [arrow] (pro3c) -| (pro4f);
\draw [arrow] (pro3d) -- (pro4g);
\draw [arrow] (pro3d) -| (pro4h);
\draw [arrow] (pro4a) -- (dec1);
\draw [arrow] (pro4b) -- (dec2);
\draw [arrow] (pro4c) -- (dec3);
\draw [arrow] (pro4d) -- (dec4);
\draw [arrow] (pro4e) -- (dec5);
\draw [arrow] (pro4f) -- (dec6);
\draw [arrow] (pro4g) -- (dec7);
\draw [arrow] (pro4h) -- (dec8);
\end{tikzpicture}    
\end{center}

\noindent Proof of Lemma \ref{lem1}:- Let $K_{2} = \mathbb{Q}(\sqrt{p_{1}p_{2}p_{3}})$ where $p_{1}, p_{2}$ and $p_{3}$ are prime numbers with $p_{1} \equiv p_{2} \equiv 3 \pmod 4$, $p_{3} \equiv 1 \pmod 4$ and such that $p_{1}p_{2}p_{3} \equiv 1 \pmod 8$ and let $u_{2} = x_{2} + y_{2}\sqrt{p_{1}p_{2}p_{3}}$ be a fundamental unit of $K_{2}$. Following a similar argument used in the proof of Lemma \ref{lem0}, we observe that $N(u_{2}+1) = 2(x_{2}+1)$ with $x_{2}+1$ being even. A similar argument then shows that $(\frac{x_{2}+1}{2}) \cdot (\frac{x_{2}-1}{2}) = p_{1}p_{2}p_{3}(\frac{y_{2}}{2})^{2}$ with $\gcd(\frac{x_{2}+1}{2}, \frac{x_{2}-1}{2}) = 1$. Therefore, $$\frac{x_{2}+1}{2} \in \{m^{2}, p_{1}m^{2}, p_{2}m^{2}, p_{3}m^{2}, p_{1}p_{2}m^{2}, p_{2}p_{3}m^{2}, p_{3}p_{1}m^{2}, p_{1}p_{2}p_{3}m^{2}\},$$ where $m$ and $n$ are positive integers with $mn = \frac{y_{2}}{2}$.

\smallskip

Now, we claim that $\frac{x_{2}+1}{2} \not \in \{m^{2}, p_{1}p_{2}p_{3}m^{2}\}$. If $\frac{x_{2}+1}{2} = m^{2}$, then we have $1 = (\frac{x_{2}+1}{2}) - (\frac{x_{2}-1}{2}) = m^{2} - p_{1}p_{2}p_{3}n^{2}$ which implies $m+n\sqrt{p_{1}p_{2}p_{3}} \in \mathcal{O}_{K}^{*}$, a contradiction to the fact that $x_{2} + y_{2}\sqrt{p_{1}p_{2}p_{3}}$ is the fundamental unit of $K{2}$. Similarly, if $\frac{x_{2}+1}{2} = m^{2}p_{1}p_{2}p_{3}$, then we obtain $1 = (\frac{x_{2}+1}{2}) - (\frac{x_{2}-1}{2}) = m^{2}p_{1}p_{2}p_{3} - n^{2}$ which implies that $\Big(\frac{-1}{p_{1}}\Big) = 1$, a contradiction again as $p_{1} \equiv 3 \pmod 4$.

\smallskip

Next, we assume the hypotheses of $(1)$. That is, $\Big(\frac{p_{1}}{p_{2}}\Big) = \Big(\frac{p_{1}}{p_{3}}\Big) = 1$. Since $p_{1} \equiv p_{2} \equiv 3 \pmod {4}$ and $p_{3} \equiv 1 \pmod {4}$, it follows from the quadratic reciprocity law that $\Big(\frac{p_{2}}{p_{1}}\Big) = -1$ and $\Big(\frac{p_{3}}{p_{1}}\Big) = 1$. We claim that $\frac{x_{2}+1}{2} \not \in \{p_{2}m^{2}, p_{1}p_{2}m^{2}, p_{2}p_{3}m^{2}\}$. If not, then we first assume that $\frac{x_{2}+1}{2} = p_{2}m^{2}$. Then $1 = \frac{x_{2}+1}{2} - \frac{x_{2}-1}{2} = p_{2}m^{2} - p_{1}p_{3}n^{2}$ implies that $\Big(\frac{p_{2}}{p_{1}}\Big) = 1$, a contradiction. Similarly, $\frac{x_{2}+1}{2} = p_{1}p_{2}m^{2}$ implies that $p_{1}p_{2}m^{2} - p_{3}n^{2} = 1$. It follows that $\Big(\frac{p_{3}}{p_{1}}\Big) = -1$, a contradiction. Again, $\frac{x_{2}+1}{2} = p_{2}p_{3}m^{2}$ implies $p_{2}p_{3}m^{2} - p_{1}n^{2} = 1$. Hence $\Big(\frac{p_{1}}{p_{2}}\Big) = -1$, a contradiction. Therefore, we conclude that $\frac{x_{2}+1}{2} \in \{p_{1}m^{2}, p_{3}m^{2}, p_{1}p_{3}m^{2}\}$. In addition, if $\Big(\frac{p_{2}}{p_{3}}\Big) = -1$, a similar argument as above forces $\frac{x_{2}+1}{2} \not \in \{p_{3}m^{2}, p_{1}p_{3}m^{2}\}$. Noting that $[a_{2}] = [2(x_{2}+1)] = [4 \cdot \frac{x_{2}+1}{2}] = [\frac{x_{2}+1}{2}]$, the proof of $(1)$ is complete. The rest of the cases follow a similar line of argument. $\hfill\Box$

\section{Proof of Theorem \ref{mainthm}}

The computation of $Po(K)$ for $K = \mathbb{Q}(\sqrt{q_{1}q_{2}}, \sqrt{p_{1}p_{2}p_{3}})$ requires us to understand the group $H^{1}(G, \mathcal{O}_{K}^{*})$, where $G \coloneqq {\rm{Gal}}(K/\mathbb{Q})$. Lemma \ref{lem0} implies that $$H^{1}(G, \mathcal{O}_{K}^{*}) \simeq \langle{[\Delta_{1}], [\Delta_{2}], [\Delta_{3}], [a_{1}], [a_{2}], [a_{3}]\rangle} \simeq \langle{[q_{1}], [q_{2}], [p_{1}p_{2}p_{3}], [a_{2}], [a_{3}]\rangle}.$$ From Remark \ref{rem1}, we know that $K$ is a P\'{o}lya field if and only if  $H^{1}(G, \mathcal{O}_{K}^{*})$ is generated by exactly $5$ elements, and no less. Lemma \ref{lem1} determines all possible cases for $a_{2}$. For each of those cases, we now choose suitable cases for $a_{3}$ such that $K$ is a P\'{o}lya field. 

\smallskip

Let $u_{3}=x_{3}+y_{3}\sqrt{p_{1}p_{2}p_{3}q_{1}q_{2}}$ be a fundamental unit of $K_{3} \coloneqq \mathbb{Q}(\sqrt{p_{1}p_{2}p_{3}q_{1}q_{2}})$. A calculation similar to the one used in the proofs of Lemma \ref{lem0} and Lemma \ref{lem1} ensures that $$N(u_{3}+1) = 2(x_{3}+1), \text{ } x_{3}+1 \text{ is even, and } \Big(\frac{x_{3}+1}{2}\Big) \Big(\frac{x_{3}-1}{2}\Big) = p_{1}p_{2}p_{3}q_{1}q_{2} \Big(\frac{y}{2}\Big)^{2}.$$ 
As $\gcd(\frac{x_{3}+1}{2}, \frac{x_{3}-1}{2}) = 1$, we have $\frac{x_{3}+1}{2} = t_{1}m^{2}, \frac{x_{3}-1}{2} = t_{2}n^{2}$ such that $t_{1}t_{2} = p_{1}p_{2}p_{3}q_{1}q_{2}$ and $mn = \frac{y}{2}$.
This, in turn, implies
\begin{equation}\label{eq1}
 t_{1}m^{2} - t_{2}n^{2} = 1.    
\end{equation}

\smallskip

Now, we prove that $K$ is a P\'{o}lya field in each of the five cases mentioned in the statement of Theorem \ref{mainthm}. Since the proofs in all the cases are essentially same, we furnish a complete proof only for the first case.

\smallskip

Since $\Big(\frac{p_{1}}{p_{2}}\Big) = \Big(\frac{p_{1}}{p_{3}}\Big) = -\Big(\frac{p_{2}}{p_{3}}\Big) = 1$, from Lemma \ref{lem1} we have $[a_{2}] = [p_{1}]$. Hence, $H^{1}(G, \mathcal{O}_{K}^{*}) \simeq \langle{[q_{1}], [q_{2}], [p_{1}], [p_{2}p_{3}], [a_{3}]\rangle}$. We now prove that the assumptions on the Lengendre symbols imply that the square-free representative of $[a_{3}]$ is divisible by $p_{2}$ or ${p_{3}}$ but not both. This, in turn, implies that $H^{1}(G, \mathcal{O}_{K}^{*}) \simeq \langle{[q_{1}], [q_{2}], [p_{1}], [p_{2}], [p_{3}]\rangle}$, and as a consequence, $K$ is a P\'{o}lya field.

\smallskip

By following the technique used in the proof of Lemma \ref{lem0} and Lemma \ref{lem1}, we note that $\big(\frac{p_{1}}{p_{3}}\big) = \big(\frac{q_{1}}{p_{3}}\big) = \big(\frac{q_{2}}{p_{3}}\big) = 1 = - \big(\frac{p_{2}}{p_{3}}\big)$ forces $\frac{x_{3}+1}{2} \not \in \{q_{1}m^{2}, q_{2}m^{2}, p_{1}q_{1}m^{2}, p_{1}q_{2}m^{2}, p_{2}p_{3}q_{1}, p_{2}p_{3}q_{2}, p_{1}p_{2}p_{3}q_{1}m^{2}, p_{1}p_{2}p_{3}q_{2}m^{2} \}$. Now the condition $\big(\frac{q_{1}q_{2}}{p_{2}}\big) = -1$ implies $\frac{x_{3}+1}{2} \not \in \{q_{1}q_{2}m^{2}, p_{1}q_{1}q_{2}m^{2}\}$. Also, $\big(\frac{p_{1}}{q_{1}}\big) = -1$ implies that $\frac{x_{3}+1}{2} \neq p_{1}$. Again, $\big(\frac{p_{1}}{q_{1}}\big) = \big(\frac{p_{1}}{q_{2}}\big)$ implies $\big(\frac{q_{1}q_{2}}{p_{1}}\big) = 1$ and therefore, $\frac{x_{3}+1}{2} \neq p_{1}p_{2}p_{3}$. Lastly,  $\frac{x_{3}+1}{2} \neq p_{2}p_{3}$ because $\big(\frac{p_{2}p_{3}}{p_{1}}\big) = -1$ and $\frac{x_{3}+1}{2} \neq p_{2}p_{3}q_{1}q_{2}$ follows from $\big(\frac{-p_{1}}{p_{2}}\big) = -1$. This eliminates all the cases for $t_{1}$ in (\ref{eq1}) such that $\frac{x_{3}+1}{2} = t_{1}m^{2}$ with $1 < \gcd(t_{1}, p_{1}p_{3}) < p_{1}p_{3}$. Consequently, the square-free representative of $[a_{3}]$ is divisible by either $p_{2}$ or $p_{3}$ but not both. This completes the proof of $(1)$ of Theorem \ref{mainthm}.

Moreover, if $\Big(\frac{p_{1}}{p_{3}}\Big) = \Big(\frac{p_{2}}{p_{3}}\Big) = 1$, then Lemma \ref{lem1} implies that either $a_{2} \in \{p_{1}, p_{3}, p_{1}p_{3}\}$ or $a_{2} \in \{p_{2}, p_{3}, p_{2}p_{3}\}$. This, in turn, implies $H^{1}(G, \mathcal{O}_{K}^{*})$ is generated by at least four elements in both the cases and consequently, either $K$ is a P\'{o}lya field or $Po(K) \simeq \mathbb{Z}/2\mathbb{Z}$. $\hfill\Box$ 

\begin{remark}\label{rmk2}
    The first part of Theorem \ref{mainthm} fails to give a precise criterion of $K$ being a P\'{o}lya field. This is evident from Lemma \ref{lem1}, due to multiple values of $a_{2}$. We want to note that, with some additional conditions, if the value of $a_{2}$ becomes unique, then for each of those five cases, one of the conditions from Theorem \ref{mainthm} will act as a sufficient condition for $K$ to be a P\'{o}lya field.  
\end{remark}
\begin{ack}
The authors thank their respective institutions for providing excellent facilities to carry out this work. The first author's reasearch is supported by  CSIR (File no: 09/1026(0036)/2020-EMR-I).
\end{ack}

\end{document}